\documentclass[a4paper,12pt]{amsart}
\usepackage[utf8]{inputenc}
\usepackage[english]{babel}
\usepackage{vmargin}

\usepackage{amsmath,amsthm,amsfonts,amssymb,latexsym, mathrsfs, yfonts,dsfont}
\usepackage{ stmaryrd }
\usepackage{multicol}
\usepackage{tikz-cd}
\usepackage{soul}

\setpapersize{A4}
\setmargins{2cm}      
{1.5cm}                        
{16.5cm}                      
{23.42cm}                    
{10pt}                           
{1.5cm}                           
{0pt}                             
{2cm}                           

\makeatletter
\def\newrefformat#1#2{%
  \@namedef{pr@#1}##1{#2}}
\def\prettyref#1{\@prettyref#1:}
\def\@prettyref#1:#2:{%
  \expandafter\ifx\csname pr@#1\endcsname\relax%
    \PackageWarning{prettyref}{Reference format #1\space undefined}%
    \ref{#1:#2}%
  \else%
    \csname pr@#1\endcsname{#1:#2}%
  \fi%
}
\makeatother

\newrefformat{eq}{\textup{(\ref{#1})}}
\newrefformat{cha}{Chapter \ref{#1}}
\newrefformat{sec}{Section \ref{#1}}
\newrefformat{tab}{Table \ref{#1} on page \pageref{#1}}
\newrefformat{fig}{Figure \ref{#1} on page \pageref{#1}}

\makeatletter
\def\indsym#1#2{%
  \setbox0=\hbox{$\m@th#1x$}%
  \kern\wd0%
  \hbox to 0pt{\hss$\m@th#1\mid$\hbox to 0pt{$\m@th#1^{#2}$}\hss}%
  \lower.9\ht0\hbox to 0pt{\hss$\m@th#1\smile$\hss}%
  \kern\wd0} 
\def\nindsym#1#2{%
  \setbox0=\hbox{$\m@th#1x$}%
  \kern\wd0%
  \hbox to 0pt{\mathchardef\nn="3236\hss$\m@th#1\nn$\kern1.4\wd0\hss}
  \hbox to 0pt{\hss$\m@th#1\mid$\hbox to 0pt{$\m@th#1^{#2}$}\hss}%
  \lower.9\ht0\hbox to 0pt{\hss$\m@th#1\smile$\hss}%
  \kern\wd0}

\makeatother

\makeatletter
\theoremstyle{plain}
\newtheorem{thm}{Theorem}[section]
\numberwithin{equation}{section} 
\numberwithin{figure}{section} 
\theoremstyle{plain}
\theoremstyle{definition}
\newtheorem{defn}[thm]{Definition}

\theoremstyle{plain}
\newtheorem{lem}[thm]{Lemma} 
\theoremstyle{plain}
\theoremstyle{plain}
\newtheorem{fact}[thm]{Fact}
\theoremstyle{plain}
\newtheorem{notation}[thm]{Notation}
\theoremstyle{plain}

\newtheorem{rem}[thm]{Remark}
\newtheorem{ej}[thm]{Example}

\newtheorem{preg}[thm]{Question}

\newtheorem{teo}[thm]{Theorem}

\theoremstyle{plain}

\def\Q{{\mathbb Q}}

\def\modK{{\mathcal K}}
\def\modL{{\mathcal L}}
\def\modM{{\mathcal M}}
\def\modN{{\mathcal N}}

\def\forkindep{\mathrel{\raise0.2ex\hbox{\ooalign{\hidewidth$\vert$\hidewidth\cr\raise-0.9ex\hbox{$\smile$}}}}}
\def\notind#1#2{#1\setbox0=\hbox{$#1x$}\kern\wd0
\hbox to 0pt{\mathchardef\nn=12854\hss$#1\nn$\kern1.4\wd0\hss}
\hbox to 0pt{\hss$#1\mid$\hss}\lower.9\ht0 \hbox to 0pt{\hss$#1\smile$\hss}\kern\wd0}

\title{General real-valued theories with the Schr\"oder-Bernstein property are stable}

\author{Alexander Berenstein, Nicol\'as Cuervo Ovalle, and Isaac Goldbring}

\address{Department of Mathematics\\ Los Andes University, Cra. 1 \#18a-12, edificio H, Bogotá, Colombia, 111711.}
\urladdr{https://pentagono.uniandes.edu.co/~aberenst/index.html}

\address{Department of Mathematics\\ Los Andes University, Cra. 1 \#18a-12, edificio H, Bogotá, Colombia, 111711.}
\email{n.cuervo10@uniandes.edu.co}
\thanks{The second-named author was partially supported by NSF grant DMS-2054477. He would also like to thank the UC Irvine Department of Mathematics for their hospitality.}

\address{Department of Mathematics\\University of California, Irvine, 340 Rowland Hall (Bldg.\# 400),
Irvine, CA 92697-3875}
\email{isaac@math.uci.edu}
\urladdr{http://www.math.uci.edu/~isaac}
\thanks{The third-named author was partially supported by NSF grant DMS-2054477.}

\date{}

\thanks{ }

\begin{document}

\vspace{-0.7cm}\begin{abstract}
We show that every general theory  \`a la Keisler with the Schr\"oder-Bernstein property is stable.  This generalizes the corresponding result from classical logic due to John Goodrick.  Our proof uses the classical result (generalized to the case that the instability is witnessed by an infinitary formula) together with a discretization technique introduced by Keisler and the third-named author.  We speculate on how our techniques could be adapted to show that every continuous theory with the Schr\"oder-Bernstein property is stable.
\end{abstract}

\maketitle
\section{Introduction}
The Schröder-Bernstein Theorem states that whenever $A$ and $B$ are two sets for which there exist injective functions $f:A\to B$ and $g:B\to A$, then there exists a bijective function $h:A\to B$. It is natural to ask when this property can be extended to a class of objects for which the existence of injective morphisms implies the existence of an isomorphism. Based on this idea, for a given language $\modL$, we say that an $\modL$-theory $T$ has the \textit{Schröder-Bernstein property} (or \textit{SB-property} for short) if, whenever two models $\modM,\modN\models T$ admit elementary embeddings $\phi:\modM\to\modN$ and $\psi:\modN\to\modM$, then there exists an isomorphism $\Phi:\modM\to\modN$.

When $\modL$ is a first-order discrete language, the SB-property has been studied extensively by different authors (see \cite{SBJohnthesis,SBpropertyWeak,SBpropertyJohn,SBpropertyW-stable,MutualEmbedding}). On the other hand, when $\modL$ is a first-order continuous language, this property was first studied in \cite{ArBeCu}.

An important result about theories with the SB property is that they are stable, a result first proved by Goodrick in \cite{SBJohnthesis}.  In fact, Goodrick showed that if any $\mathcal{L}_{\infty,\omega}$ formula in $T$ has the order property, then $T$ does not have the SB-property; see also Fact \ref{DiscreteUnstableSB} below.

It is thus natural to ask whether or not the same result holds in continuous logic:

\begin{preg}\label{mainquestion}
If $T$ is a complete, continuous theory with the SB-property, must $T$ be stable?
\end{preg}

While we cannot settle this question at the moment, in this paper we prove a related result, namely that the version of the above question for Keisler's general real-valued logic \cite{keisler} has a positive solution:

\begin{teo}\label{MainTeo}
If $T$ is a complete \emph{general} theory with the SB-property, then $T$ is stable.
\end{teo}

A brief discussion of Keisler's general real-valued logic is given in Subsection 2.1 below and its connection to continuous logic is discussed in the final section.

In order to show that classical theories with the SB-property are stable, Goodrick  \cite{SBJohnthesis} uses as a starting point an argument by Shelah, to construct, for a complete, countable, unstable theory $T$, a family of $2^\kappa$ non-isomorphic models of cardinality $\kappa$ for all infinite cardinals $\kappa>\aleph_0$ (see \cite[Theorem VIII.3.2]{Shelah})\footnote{Goodrick does not assume that the theories under consideration are countable.}. The changes introduced in the construction presented \cite{SBJohnthesis} ensure that the non-isomorphic models in a subfamily of size $2^\kappa$ are also elementarily bi-embeddable.

Instead of reproducing the arguments from \cite{SBJohnthesis} in the continuous setting, 
we use the models constructed in the proof in \cite{SBJohnthesis}, along with a ``discretization'' method introduced by Goldbring and Keisler in \cite{GoKei} which associates, to each general theory $T$ in a language $\mathcal{L}$, a discrete theory $T_\downarrow$ in the language $\mathcal{L}_\downarrow$. We then argue as folllows: in order to obtain a contradiction, assume that $T$ is a complete, unstable general theory with the SB-property. We construct an $(\modL_\downarrow)_{\infty,\omega}$-formula $\psi(x,y)$ with the order property in $T_\downarrow$. Consequently, by Goodrick's result above, the associated theory $T_\downarrow$ (or, more precisely, a certain completion of it) does not have the SB-property. Using the collection of models constructed in \cite{SBJohnthesis} that witness the failure of the SB-property for $T_\downarrow$, and with some extra work, we show how to construct elementarily bi-embeddable models of $T$ that are not isomorphic, contradicting that $T$ has the SB-property. We stress that we make crucial use of the fact that the statement Goodrick's result holds even if the instability is witnessed by an infinitary formula.


This paper is organized as follows. In Section 2, we recall some basic notions and review some preliminary results that we will use in the paper:  in Subsection 2.1, we summarize Keisler's general real-valued logic, in Subsection 2.2 we summarize and further expand upon Goldbring and Keisler's technique of discretizing a general structure, and in Subsection 3.3 we summarize Goodrick's proof that classical unstable theories do not have the SB-property. In Section 3, we prove our main theorem.  In Section 4, we offer two possible avenues for adapting our proof to settle Question \ref{mainquestion}.

\section{Preliminaries}
In this section, we present some results and definitions needed in the next section to prove the main theorem. We divide the section into three parts. First, we review Keisler's theory of general real-valued structures.  Then we recall the “discretization” method presented in \cite{GoKei}, pointing out some additional facts that will be used later on in the paper.  Finally, we recall some important aspects of Goodrick's proof that classical unstable theories do not have the SB-property, as presented in \cite{SBJohnthesis}. 

We start by reminding the reader of the definition of an infinitary formula in $\modL_{\infty,\omega}$, as it plays an important role in the sequel. 

\begin{defn}
For any first order, discrete language $\modL$ and any infinite cardinal $\kappa$, the logic $\modL_{\kappa,\omega}$ is the smallest set of formulas that contains all formulas in $\modL$, is closed under Boolean operations, and for any index set $I$ with $|I|<\kappa$ and  any set $\{\varphi_i(\vec x)\ |\ i\in I\}$ of formulas in $\modL_{\kappa,\omega}$, where $\vec x$ has finite length, the conjunction $\bigwedge_{i\in I}\varphi_i(\vec x)$ is in $\modL_{\kappa,\omega}$. We write
$\modL_{\infty,\omega}$ for the union of the $\modL_{\kappa,\omega}$'s as $\kappa$ ranges over all cardinals.
\end{defn}

\subsection{General real-valued structures}

In this paper, we will use three slightly different model-theoretic contexts, and we take the opportunity here to advise the reader as to the subtle differences.

The first context is that of Keisler's general real-valued logic from \cite{keisler}.\footnote{For simplicity, in this paper, we restrict to the $[0,1]$-valued version of this logic.}  Here, a language $\mathcal{L}$ consists of a set of predicate, function, and constant symbols; this is the same as in usual first-order logic.  One forms $L$-terms as usual and \emph{general atomic $L$-formulae} are constructed by applying predicate symbols to terms; \emph{we note that, in general logic, unlike in classical logic, there is not a nological symbol for equality}.  One obtains more \emph{general formulae} by plugging them into connectives which, like in continuous logic, are continuous functions $[0,1]^n\to [0,1]$ as $n$ varies.  Finally, the quantifiers of this logic are $\sup$ and $\inf$, just as in continuous logic.  A \emph{general $L$-theory} is a set of general $L$-sentences.

We emphasize an important point:  although the notion of a language is the same in both general logic and first-order logic, the corresponding notion of formula is quite different (both in what we consider to be an atomic formula and in the construction of more complicated formulae from pre-existing formulae).  As a result, given a language $\mathcal{L}$, we will emphasize the difference between general $\mathcal{L}$-formulae/theories and first-order $\mathcal{L}$-formulae/theories. 

A \emph{general $\mathcal{L}$-structure} interprets the function and constant symbols just as in classical logic, but interprets the predicate symbols as functions from the appropriate cartesian power of the universe to $[0,1]$.  Interpretations of formulae are as expected, for details see \cite{keisler}.  Once again, given a language $\mathcal{L}$, we will explicitly distinguish between general $\mathcal{L}$-structures and first-order $\mathcal{L}$-structures.

Given a general $\mathcal{L}$-structure $\mathcal{K}$, one defines the notion of \emph{Leibniz equality} $\doteq$ on $\mathcal{K}$ by declaring $a\doteq b$ if and only if, for all atomic formulae $\varphi(x,y)$ and all $c\in \mathcal{K}$, we have $\varphi^{\mathcal{K}}(a,c)=\varphi^{\mathcal{K}}(b,c)$.  If $\doteq$ coincides with ordinary equality, we call $\mathcal{K}$ a \emph{reduced} $\mathcal{L}$-structure.  Given an arbitrary $\mathcal{L}$-structure, one defines its \emph{reduction} to be the obvious quotient of $\mathcal{K}$ by $\doteq$ (see \cite[Subsection 2.1]{keisler} for details).

In the last section of this paper, we explain how continuous logic fits in as a special case of general real-valued logic.

\subsection{Discretization}

Throughout this subsection, $\modL$ denotes an arbitrary language.

\begin{defn}
Let $\modL_{\downarrow}$ be the language with the same function and constant symbols as $\modL$ and with an $n$-ary predicate symbol $P_{\leq r}$ for each $n$-ary predicate symbol $P$ of $\modL$ and rational $r\in[0,1)$.  For a reduced $\modL$-structure $\modM$, let $\modM_\downarrow$ be the first-order $\modL_\downarrow$-structure with the same universe, functions, and constants as $\modM$, and such that for each $n$-ary predicate symbol $P$, each rational $r\in [0,1)$, and each $\vec a\in M^n$, we have $\modM_\downarrow\models P_{\leq r}(\vec a)$ if and only if $P^{\modM}(\vec a)\leq r$. 

We say that a first-order $\mathcal{L}_\downarrow$-structure $\modK$ is \textit{increasing} if for every predicate $P$ of $\modL$ and rational $r\leq s$ in $[0,1)$, we have $\modK\models (\forall \vec x)[P_{\leq r}(\vec x)\Rightarrow P_{\leq s}(\vec x)].$

If $\mathcal{K}$ is a first-order $\mathcal{L}_\downarrow$-structure, let $\modK_\uparrow$ be the reduction of the general $\mathcal{L}$-structure  $\modN$ with the same universe, functions and constants as $\modK$, and such that for each $n$-ary predicate symbol $P$ of $\modL$ and $\vec a\in K^n$, $P^\modN(\vec a)=\inf\{s\in \mathbb{Q}\cap [0,1) \  : \ \modK\models P_{\leq s}(\vec a)\}.$
\end{defn}

\begin{fact}[Lemma 5.5 in \cite{GoKei}]\label{downup}
For each reduced $\mathcal{L}$-structure $\modM$, we have that $\modM_\downarrow$ is increasing and $\modM=\modM_{\downarrow\uparrow}.$
\end{fact}

\begin{defn}
For a general $\modL$-theory $T$, let $T_\downarrow$ be the first-order $\modL_\downarrow$-theory of the class of all increasing $\modL_\downarrow$-structures $\modK$ such that $\modK_\uparrow\models T$.
\end{defn}

\begin{fact}[Lemma 5.7 in \cite{GoKei}] 
Let $T$ be a general $\modL$-theory. Then:
\begin{itemize}
    \item[i)] If $\modM\models T$ is reduced, then $\modM_\downarrow\models T_\downarrow$. 
    \item[ii)] Every first-order model of $T_\downarrow$ is increasing. 
    \item[(iii)]\label{updownrule} For every first-order $\modL_\downarrow$-structure $\modK$, we have $\modK\models T_\downarrow$ if and only if $\modK_\uparrow\models T$.
\end{itemize}
\end{fact}

Although going down and then up returns back to the structure we started with, this is not necessarily the case with going up and then down:

\begin{ej}\label{standard}
Suppose that $\modM$ is the (reduced) general structure whose underlying universe is the interval $[0,1]$ and with a predicate $P$ such that $P^{\modM}(x)=x$ for all $x\in [0,1]$.  Let $\mathcal{K}:=\mathcal{M}_\downarrow$ and let $\mathcal{K}'$ be the first-order $\mathcal{L}_\downarrow$ structure that is just like $\mathcal{K}$ except that $\mathcal{K}'\not\models P_{\leq 0}(0)$.  Then $\mathcal{K}'_\uparrow=\mathcal{M}$, whence $\mathcal{K}'_{\uparrow\downarrow}=\mathcal{K}\not=\mathcal{K}'$.
\end{ej}

The previous example notwithstanding, sometimes going up and then down does gives us back the original structure.  Essentially, there are two hurdles one needs to avoid:  1)  ensuring that the reduction process is unnecessary when going up, and 2) ensuring that the ``incompleteness'' phenomenon appearing in the previous example does not occur.  Towards this end, let $\sigma_1$ and $\sigma_2$ denote the following two first-order $(\modL_\downarrow)_{\infty,\omega}$-sentences:

$$\forall x,y\left(x\not=y\to \bigvee_{\varphi(u,v)\in \modL_{at}}\bigvee_{r\in \mathbb{Q}^{>0}\cap [0,1)}\exists z \neg[\varphi_{\leq r}(x,z) \leftrightarrow \varphi_{\leq r}(y,z)]\right)$$
$$\bigwedge_{Q\in \mathcal{L}}\bigwedge_{r\in\Q^{>0}\cap [0,1)}\left((\forall\Vec{x})\left(\left[\bigwedge_{s> r}Q_{\leq s}(\vec x)\right]\to Q_{\leq r}(\vec x)\right) \right).$$
Here, in the definition of $\sigma_1$, the first infinite disjunction is over all atomic formulae in two variables, and the notation $\varphi_{\leq r}$ denotes $P_{\leq r}(t_1,\ldots,t_n)$, where $P$ is a predicate symbol of $\mathcal{L}$, $t_1,\ldots,t_n$ are $\mathcal{L}$-terms, and $\varphi=P(t_1,\ldots,t_n)$.

The following lemma is clear:

\begin{lem}
For first-order $\modK\models T_\downarrow$, we have that $\modK\models \sigma_1$ if and only if, in the formation of $\modK_\uparrow$, no reduction procedure is needed.

\end{lem}
Call a model of $T_\downarrow$ \textbf{standard} if it is of the form $M_\downarrow$ for some reduced $M\models T$.  Finally set $\sigma:=\sigma_1\wedge \sigma_2$.

\begin{lem}\label{DownLemma}
For first-order $K\models T_\downarrow$, the following are equivalent:
\begin{enumerate}
\item $\modK$ is standard.
\item $\modK\models \sigma$.
\item $\modK=\modK_{\uparrow\downarrow}$.
\end{enumerate}
\end{lem}

\begin{proof}
(1) implies (2) is clear from the definition and (3) implies (1) follows from the fact that $M_{\downarrow\uparrow}=M$. Finally, the direction (2) implies (3) follows from unraveling the meaning of $\sigma$.
\end{proof}

In general, $T_\downarrow$ need not be a complete first-order theory.  For example, in Example \ref{standard}, setting $T:=\operatorname{Th}(\mathcal{M})$, we have that $\mathcal{K},\mathcal{K}'\models T_\downarrow$ although $\mathcal{K}\not\equiv\mathcal{K}'$ since they disagree on the truth of the sentence $\exists xP_{\leq 0}(x)$.  Nevertheless, one has the following:

\begin{fact}\label{completefact}(Lemma C.4 in \cite{GoKei})
If $\modM,\modM'\models T$ are $\aleph_1$-saturated general models of $T$, then $\modM_\downarrow\equiv \modM'_\downarrow$.
\end{fact}

\subsection{The discrete version of Goodrick's result}

\begin{defn}
We say that a first order discrete $\modL$-theory $T$ has the $\modL_{\infty,\omega}$ \textit{order property} if there is a $\modL_{\infty,\omega}$-formula $\varphi(\vec x,\vec y)$ such that for any linear ordering $I$, there is a model $\modM$ of $T$ and a set $(\vec a_i)_{i\in I}$ of tuples in $M$ such that $\modM\models \varphi(\vec a_i,\vec a_j)$ if and only if $i<j$. 
\end{defn}

\begin{rem}
The first order, discrete, theories $T$ with the  $\modL_{\infty,\omega}$ order property witnessed by an $\modL_{\omega,\omega}$-formula are precisely the unstable theories. 
\end{rem}

We now summarize the construction used by Goodrick in \cite{SBJohnthesis} to prove the discrete version of Theorem \ref{MainTeo}.

\begin{defn}[Definition 2.2.2 in \cite{SBJohnthesis}]
\noindent\begin{itemize}
    \item[i)] If $(I,<)$ is a linear order and $J\subseteq I$, then for $\vec a,\vec b\in I$, we use the abbreviation $\vec a\sim\vec b \mod J$ to mean that $\vec a$ and $\vec b$ have the same quantifier-free type in the discrete language of orders over the parameter set $J$.
    \item[ii)] Let $\modM$ be a discrete structure. If $(I,<)$ is an order, then an $I$-indexed sequence $(\vec a_s)_{s\in I}$ of finite tuples from $M$ is \textit{skeletal in} $\modM$ if it is indiscernible, and for any $\vec c\in M$, there is a finite $J\subseteq I$ such that if $\vec s,\vec t\in I$ and $\vec s\sim\vec t\mod J$, then for any $\varphi(\vec x,\vec y)$ in the language $\mathcal{L}_{\infty \omega}$ with the right number of variables, $\modM\models \varphi(\vec a_{\vec s},\vec c)\leftrightarrow\varphi(\vec b_{\vec t},\vec c).$
\end{itemize}
\end{defn}

Consider a first order, discrete, theory $T$ with the $\modL_{\infty,\omega}$-order property, witnessed by a formula $\psi(\vec x,\vec y)$, where both $\vec x$ and $\vec y$ have length $m$.  Let $\modL^1\supset\modL$ be an expansion of the language by a single $2m$-ary relation $R(\vec x,\vec y)$ and let $\mathbf K^1$ be the class of all models of $T$
expanded to $\modL^1$ by interpreting $R(\vec x,\vec y)$ by $\psi(\vec x,\vec y)$. If $T'$ is a theory in a language $\modL'\supseteq\modL$ and $\Gamma$ is a set of types in $T'$, then $\operatorname{PC}_{\modL}(T',\Gamma)$ is the projective class corresponding to all reducts to $\modL$ of models of $T'$ that omit all types in $\Gamma$. 

\begin{fact}[Fact 2.3.1 in \cite{SBJohnthesis}]
There is a first order, discrete, language $\modL^2\supseteq\modL^1$, a theory $T^2$ in $\modL^2$ with the following properties: 
\begin{itemize}
    \item[i)] $T^2\supseteq \operatorname{Th}_{\modL^1}(\mathbf K^1)$.
    \item[ii)] $T^2$ has Skolem functions.
    \item[iii)] For each subformula $\varphi(\vec z)$ of $\psi$, there is a relation $R_\varphi(\vec z)\in\modL^2$.
    \item[iv)] $\mathbf K^1=\operatorname{PC}_{\modL^1}(T^2,\Gamma)$, where $\Gamma$ is the set of types in $T^2$ of all types of elements in nonstandard interpretations of the $R_\varphi$'s.
    \item[v)] For every model $\modM\in \operatorname{PC}_{\modL^2}(T^2,\Gamma)$, every subformula $\varphi(\vec z)$ of $\psi$, and every $\vec a\in M$, $\modM\models R_\varphi(\vec a)$ if and only if $\modM\models\varphi(\vec a)$. 
\end{itemize}
\end{fact}

\begin{fact}[Fact 2.3.2 in \cite{SBJohnthesis}]\label{EM(I)}
    With the notation of the previous fact, there is a mapping $\operatorname{EM}^2$ from the class of all linear orderings into $\operatorname{PC}_{\modL^2}(T^2,\Gamma)$ with the following properties: 
    \begin{itemize}
        \item[i)] $|\operatorname{EM}^2(I)|\leq |T^2|+|I|$.
        \item[ii)] For each linear order $I$, the model $\operatorname{EM}^2(I)$ contains an indiscernible sequence $(\vec a_s)_{s\in I}$ of $m$-tuples which is skeletal in $\operatorname{EM}^2(I)$, called its \textit{skeleton}.
        \item[iii)] For each $I$, the skeleton $(\vec a_s)_{s\in I}$ has the property that $\operatorname{EM}^2(I)\models R(\vec a_s,\vec a_t)$ if and only if $s<t$.
        \item[iv)] For any two orders $I$ and $J$, if $(\vec a_s)_{s\in I}$ is the skeleton of $\operatorname{EM}^2(I)$, $(\vec b_t)_{t\in J}$ is the skeleton of $\operatorname{EM}^2(J)$, and $\vec s\in I$ and $\vec t\in J$ are increasing finite tuples of the same length, then $\operatorname{tp}(\vec a_{\vec s})=\operatorname{tp}(\vec b_{\vec t})$.
        \item[v)] If $J$ is a subordering of $I$, then there is an elementary embedding $f$ of $\operatorname{EM}^2(J)\restriction \modL$ into $\operatorname{EM}^2(I)\restriction\modL$. 
    \end{itemize}
    We write $\operatorname{EM}^1(I)$ for $\operatorname{EM}^2(I)\restriction\modL^1$ and $\operatorname{EM}(I)$ for $\operatorname{EM}^2(I)\restriction\modL$. Note that the skeleton of $\operatorname{EM}^2(I)$ is still skeletal in $\operatorname{EM}^1(I)$.
\end{fact}

Goodrick uses the previous set-up to establish the following:

\begin{fact}[Theorem 2.4.1 in \cite{SBJohnthesis}]\label{DiscreteUnstableSB}
If $T$  is complete and has the $\modL_{\infty,\omega}$ order property, then $T$ does not have the SB property. In fact, for any cardinal $\kappa$, there is a cardinal $\lambda>\kappa$ such that $T$ has a collection of models $(\modM_i)_{i<2^\lambda}$, each having size $\lambda$, and such that, for \emph{distinct} $i,j<2^\lambda$, $\modM_i$ and $\modM_j$ are elementarily bi-embeddable but not isomorphic.  
\end{fact}

\begin{rem}\label{JohnsArg}
Goodrick proves Fact \ref{DiscreteUnstableSB} constructing a collection of $2^\lambda$ many orders $J_\alpha$, where $\lambda$ is as in Fact \ref{DiscreteUnstableSB}, for which their associated models $\operatorname{EM}(J_\alpha)$ are pairwise elementarily bi-embeddable but nonisomorphic.
\end{rem}

\section{Main theorem}

\begin{defn}
We say that a complete general theory $T$ is \textit{unstable}\footnote{In \cite{keisler}, Keisler characterizes stability of general theories both in terms of stable independence relations and in terms of definability of types over models.  Standard arguments show that stability of general theories is further equivalent to no formula having the order property as in our definition.} if there is a formula $\psi(\vec x,\vec y)$ and model $\modM\models T$ such that there is a collection $(\vec a_n)_{n<\omega}$ of tuples in $\modM$ and real numbers $r<s$ such that $\psi(\vec a_i,\vec a_j)\leq r$ if $i<j$ and $\psi(\vec a_i,\vec a_j)\geq s$ if $i>j$.
\end{defn}

\noindent\textit{We now work towards proving Theorem \ref{MainTeo}}: 

In order to obtain a contradiction, suppose that $T$ is a complete, unstable general theory with the SB-property. Furthermore, we can Morleyize the theory and assume every formula is equivalent to a predicate in the language. As a result, there is a predicate $P(\vec x,\vec y)$ in the language that is an unstable formula in $T$ witnessed by values $r$ and $s$.  Moreover, by compactness, for any linear ordering $I$, there is $\modM\models T$ and $(a_\alpha)_{\alpha\in I}$ such that $P^{\modM}(a_\alpha,a_\beta)\leq r$ if $\alpha<\beta$ while $P^{\modM}(a_\alpha,a_\beta)\geq s$ if $\alpha\geq \beta$. Furthermore, by
modifying the formula, we may take $r=0$ and $s=1$. 

Set $S:=\operatorname{Th}(M_\downarrow)$, where $M$ is some (equiv. any) $\aleph_1$-saturated model of $T$ (see Fact \ref{completefact}). Note that $S$ is now a first order theory, while $T$ is a general theory.

Let $\psi(\vec x,\vec y)$ be $P_{\leq 1/2}(\vec x,\vec y)\wedge\sigma$. Then $\psi(\vec x, \vec y)$ is a $(\modL_{\downarrow})_{\infty,\omega}$-formula with the order property in $S$:  given any linear ordering $I$, there is an $\aleph_1$-saturated $\mathcal{M}'\models T$ and a sequence $(a_\alpha)_{\alpha\in I}$ in $\mathcal{M}'$ such that $\modM'_\downarrow\models\psi(a_\alpha,a_\beta)$ if and only if $\alpha<\beta$.

By Remark \ref{JohnsArg}, there are  uncountably many orders $J_\alpha$ such that, setting  $\modK_\alpha:=\operatorname{EM}^2(J_\alpha)\restriction\modL_\downarrow$, we have that $\modK_\alpha\models S$ and the $\modK_\alpha$'s  are elementarily bi-embeddable but pairwise nonisomorphic. 

\begin{lem}\label{DownWitness}
$\modK_\alpha\models\sigma$ for all $\alpha$.
\end{lem}

\begin{proof}
Since $\operatorname{EM}^2(J_\alpha)$ has skeleton $(b_s:s\in J_\alpha)$, by Fact \ref{EM(I)} we have that $R(b_s,b_t)$ if and only if $s<t$. In particular $\operatorname{EM}^2(J_\alpha)\models\sigma$, so $\modK_\alpha\models\sigma.$
\end{proof}

By Fact \ref{updownrule}, $(\modK_\alpha)_\uparrow\models T$.  By construction, for every pair $\alpha,\beta$, there exists an elementary embedding $\Phi_{\alpha,\beta}:\modK_\alpha\hookrightarrow \modK_\beta$. This elementary embedding induce an embedding $\Tilde\Phi_{\alpha,\beta}:(\modK_\alpha)_\uparrow\hookrightarrow(\modK_\beta)_\uparrow$,  which is elementary since $T$ has QE. By assumption, $T$ has the SB-property, thus $(\modK_\alpha)_\uparrow\cong(\modK_\beta)_\uparrow,$ and thus
$(\modK_\alpha)_{\uparrow\downarrow}\cong(\modK_\beta)_{\uparrow\downarrow}$. By Lemmas \ref{DownLemma} and \ref{DownWitness}, $\modK_\alpha\cong K_\beta$, which is a contradiction. 

\section{Towards a result for continuous theories with the SB-property}

Suppose that $\mathcal{L}$ is a continuous language and that $T$ is an $\mathcal{L}$-theory.  We may also view $\mathcal{L}$ as a general language.  In this case, a general model of $T$ is what is usually called \emph{pre-structure}, in the sense of continuous logic\footnote{As pointed out in \cite[Subsection 2.5]{keisler}, $T$ knows that the metric symbol should be interpreted as a pseudometric and that the symbols should obey the appropriate moduli of uniform continuity.}, that is also a model of $T$, whereas a reduced general model of $T$ is what is usually called a \emph{pre-complete} model of $T$, in that the interpretation of the metric symbol is indeed a metric, albeit possibly incomplete.

This poses a problem when it comes to the SB-property, for it is a priori possible that a continuous (complete) theory $T$ may have the SB-property as a continuous theory but does not have the SB-property as a general theory.  Indeed, the latter would require that any two pre-complete models of $T$ that are elementarily bi-embeddable are indeed isomorphic, which is a stronger requirement.

This leads to the following question:

\begin{preg}\label{continuousquestion}
If $T$ is a continuous theory with the SB-property, does $T$ still have the SB-property when viewed as a theory in Keisler's general real-valued logic?
\end{preg}

If the following question has a positive answer, then the main result of our paper yields a positive answer to Question \ref{mainquestion}, for stability of $T$ is independent of whether we view $T$ as a continuous theory or general theory.

Irrespective of the truth of the previous question, we offer two strategies for settling Question \ref{mainquestion}.

First, we note the following easy result:

\begin{lem}
Let $\vec x=(x_i)_{i<\omega}$ be a countably infinite tuple of variables and let $\tau$ be the following $(\mathcal{L}_\downarrow)_{\omega_1,\omega_1}$-sentence:
$$\forall \vec x\left(\bigwedge_r\bigvee_{m}\bigwedge_{k,l\geq m}(d_{\leq r}(x_k,x_l)\to \exists x\bigwedge_r\bigvee_m\bigwedge_{k\geq m}d_{\leq r}(x_k,x)\right).$$
Then for all $\modK\models T_\downarrow$, if $\modK\models \tau$, then $\modK_\uparrow$ is complete.
\end{lem}

Thus, if the analogue of Goodrick's results from Subsection 2.3 above held for $\mathcal{L}_{\infty,\omega_1}$-formulae (where one is allowed to quantify over countably many variables), then our method of proof would continue to hold, yielding a positive answer to Question \ref{continuousquestion}.  Interestingly enough, in \cite[Remark 2.4.7]{SBJohnthesis}, Goodrick himself asks if his methods extend to cover $\mathcal{L}_{\infty,\infty}$-formulae.  As far as we are aware, the possibility of this generalization holding is still an open question.

Another possible avenue is the following.  First we introduce:

\begin{notation}
Following the notation of the proof of Theorem \ref{MainTeo}, we write $\overline{\modK_\alpha}$ to denote the completion of $(\modK_\alpha)_\uparrow$.
\end{notation}

\begin{preg}
If $\modK_\alpha\not\cong \modK_\beta$, is it true that $(\overline{\modK_\alpha})_\downarrow\not\cong (\overline{\modK_\beta})_\downarrow$?
\end{preg}

Given how the structures $\modK_\alpha$ are ``controlled'' by their skeleta, it seems plausible that this control could extend to the downs of the completions.

\end{document}